\documentclass[reqno, 11pt]{amsart}
\usepackage{amssymb}
\usepackage{graphicx}
\usepackage{stmaryrd}
\usepackage[hypertex]{hyperref}

\usepackage{verbatim}
\numberwithin{equation}{section}

\newcommand{\mysection}[1]{\section{#1}
\setcounter{equation}{0}}

\newtheorem{theorem}{Theorem}[section]
\newtheorem{corollary}[theorem]{Corollary}
\newtheorem{lemma}[theorem]{Lemma}
\newtheorem{proposition}[theorem]{Proposition}

\theoremstyle{definition}
\newtheorem{remark}[theorem]{Remark}

\theoremstyle{definition}

\theoremstyle{definition}
\newtheorem{assumption}[theorem]{Assumption}

\makeatletter
\def\dashint{\operatorname%
{\,\,\text{\bf--}\kern-.98em\DOTSI\intop\ilimits@\!\!}}
\makeatother

\newcommand{\vertiii}[1]{{\vert\kern-0.25ex\vert\kern-0.25ex\vert #1  \vert\kern-0.25ex\vert\kern-0.25ex\vert}}

\def\bR{\mathbb{R}}

\def\bH{\mathbb{H}}
\def\bQ{\mathbb{Q}}

\def\bL{\mathbb{L}}
\def\bM{\mathbb{M}}

\def\cD{\mathcal{D}}

\def\cH{\mathcal{H}}

\def\cM{\mathcal{M}}

\begin{document}

\title[Neumann Problem in weighted Sobolev spaces]
{Neumann problem for non-divergence elliptic and parabolic equations with BMO$_x$ coefficients in weighted Sobolev spaces}

\author[H. Dong]{Hongjie Dong}
\address[H. Dong]{Division of Applied Mathematics, Brown University,
182 George Street, Providence, RI 02912, USA}
\email{Hongjie\_Dong@brown.edu}
\thanks{H. Dong was partially supported by the NSF under agreement DMS-1056737.}

\author[D. Kim]{Doyoon Kim}
\address[D. Kim]{Department of Applied Mathematics, Kyung Hee University, 1732 Deogyeong-daero, Giheung-gu, Yongin-si, Gyeonggi-do 446-701, Republic of Korea}

\email{doyoonkim@khu.ac.kr}

\thanks{}

\author[H. Zhang]{Hong Zhang}
\address[H. Zhang]{Division of Applied Mathematics, Brown University,
182 George Street, Providence, RI 02912, USA}
\email{Hong\_Zhang@brown.edu}
\thanks{H. Zhang was partially supported by the NSF under agreement DMS-1056737.}

\subjclass[2010]{35J25, 35K20, 35R05}

\keywords{$L_p$ estimates, weighted Sobolev spaces, parabolic equations}

\begin{abstract}
We prove the unique solvability in weighted Sobolev spaces of non-divergence form elliptic and parabolic equations on a half space with the homogeneous Neumann boundary condition. All the leading coefficients are assumed to be only measurable in the time variable and have small  mean oscillations in the spatial variables.
Our results can be applied to Neumann boundary value problems for {\em stochastic} partial differential equations with BMO$_x$ coefficients.
\end{abstract}

\maketitle

\mysection{Introduction}
In this paper, we study $L_p$ estimates for elliptic and parabolic equations in non-divergence form:
$$
a_{ij}D_{ij} u+b_iD_i u+cu-\lambda u=f
\quad \text{in}\,\,\, \bR^d_+,
$$
\begin{equation*}
-u_t+a_{ij}D_{ij} u+b_iD_i u+cu-\lambda u=f\quad \text{in}\,\,\, (-\infty,T) \times\bR^d_+,
\end{equation*}
with the homogeneous Neumann boundary condition, where $\lambda$ is a nonnegative constant and $\bR^d_+$ is a half space defined by
$$
\bR^d_+=\{x = (x_1, \cdots, x_d) = (x_1,x') : x_1>0\}.
$$
We consider the equations in weighted Sobolev spaces with measures
$$
\mu_d(dx) = x_1^{\theta-d}\, dx\quad \text{and} \quad \mu(dx\,dt)=x_1^{\theta-d}\,dx\,dt
$$ in the elliptic and parabolic cases, respectively, for some  $\theta\in (d-1,d-1+p)$.

Krylov \cite{Kry99} first studied Laplace's equation and the heat equation in weighted Sobolev spaces $H_{p,\theta}^\gamma$ and $\bH_{p,\theta}^\gamma$; see Section 2 for precise definitions.
After \cite{Kry99}, there has been quite a few work on the solvability theory for elliptic and parabolic equations in weighted Sobolev spaces, for instance, see \cite{KimKry04, KozNaz09, KimLee13, KimKimLee14}. In particular, the authors of \cite{KimLee13, KimKimLee14} studied second-order parabolic equations with the Dirichlet boundary condition in weighted Sobolev spaces with leading coefficients having small mean oscillations. The motivation of such theory came from stochastic partial differential equations (SPDEs) and is well explained in \cite{Kry94}.

Recently, Dong and Kim \cite{DK14} studied both divergence and non-divergence type elliptic and parabolic equations on a half space in weighted Sobolev spaces with the Dirichlet boundary condition. The coefficients in \cite{DK14} are contained in a larger class than those in \cite{KimLee13, KimKimLee14}. Namely, the leading coefficients are assumed to be only measurable in $t$ and $x_1$ except $a_{11}$, which is measurable in either $t$ or $x_1$, where $x_1$ is the normal direction. Kozlov and Nazarov \cite{KozNaz14} considered an oblique derivative problem for non-divergence type parabolic equations on a half space with coefficients discontinuous in $t$ (and continuous in $x$) in a weighted Sobolev space.  Their proof is based on a careful investigation of Green's functions. In this paper, we extend the result in \cite{KozNaz14} to a more general setting. Namely, the coefficients considered in this paper are measurable in the time variable and have small mean oscillations with respect to a weighted measure in the spatial variables. We call this class of coefficients BMO$_x$. The weight, for instance, for the parabolic case  is $\mu(dx\,dt)=x_1^{\theta-d}\,dx\,dt$, where $\theta\in(d-1,d-1+p)$. The condition $\theta\in (d-1,d-1+p)$ is sharp even for the heat equation;  see \cite{Kry99}.
We note that the coefficients $a_{ij}$ in \cite{DK14} also have small mean oscillations with respect to a weighted measure as functions of $x' \in \bR^{d-1}$ (whereas, in this paper as functions of $x \in \bR^d$), but the size of the modulus of regularity of $a^{ij}$ is proportional to the distance to the boundary.
See Assumption \ref{assumption} and Remark \ref{rem0717_1}.

Since the counterexamples of Ural'ceva \cite{Ura67} and Nadirashvili \cite{Nad97}, particular types of discontinuous coefficients have been considered for the solvability of equations. One type of discontinuous coefficients, which has been widely considered, is the class of vanishing mean oscillation (VMO) coefficients. The study of equations with VMO coefficients was initiated by Chiarenza, Frasca, and Longo \cite{CFL91, CFL93}. In  \cite{Kry97} Krylov gave a unified approach to investigating parabolic and elliptic equations in unweighted Sobolev spaces with coefficients that are measurable in the time variable and  have small mean oscillations with respect to the usual Lebesgue measure in the spatial variables (BMO$_x$ with respect to the Lebesgue measure); see also \cite{kry07}.
In fact, the coefficients in \cite{Kry97} are called VMO$_x$ coefficients, but their mean oscillations in $x$ do not have to vanish as the radii of cylinders go to zero.
For more related work about $L_p$ theory with BMO$_x$ or partially BMO$_x$ coefficients for parabolic systems and higher-order parabolic systems, we refer the reader to \cite{DK11, DK11ARMA, Dong10, HZ14} and the references therein.

Our proof is in the spirit of the approach by Krylov \cite{Kry97}.  The key point of such approach is to establish mean oscillation type estimates for equations with simple coefficients, i.e.,  coefficients are only measurable as functions of $t$. Then we apply a perturbation argument, which is well suited to the mean oscillation estimates, to deal with BMO$_x$ coefficients. Finally we obtain the  desired $L_p$ estimates by applying the Fefferman--Stein theorem on sharp functions and the Hardy--Littlewood maximal function theorem in weighted $L_p$ spaces.

Here one of the main steps is to get the mean oscillation estimates of $D^2u$. For a simple equation
\begin{equation*}
-u_t+a_{ij}(t) D_{ij}u=f
\end{equation*}
in $\bR \times \bR^d_+$ with the Neumann boundary condition $D_1u=0$ on $\{x_1=0\}$, we treat $DD_1u$ and $D^2_{x^\prime}u$ separately.  We estimate $DD_1u$ as follows. Differentiating the equation above with respect to $x_1$, it is easily seen that $D_1u$ satisfies the divergence type parabolic equation
\begin{equation*}
-(D_1u)_t+D_i(a_{ij}D_j(D_1u))=D_1 f
\end{equation*}
in $\bR \times \bR^d_+$ with $D_1 u=0$ on $\{x_1=0\}$.
Therefore, we can apply a result in \cite{DK14} to obtain the mean oscillation estimates of $DD_1u$. On the other hand, the estimates of $D^2_{x^\prime}u$ are much involved. We treat the mean oscillations of $D_{x^\prime}^2u$ in the $x_1$ variable and $x^\prime$ variables differently. By integrating by parts and the Poincar\'e inequality in weighted spaces, we manage to bound the mean oscillations of $D_{x^\prime}^2u$ in the $x_1$ variable by the maximal functions of $DD_1u$.
For the mean oscillations in $x^\prime$ variables, we write the equation in the following form
\begin{equation*}
-u_t+\sum_{i,j\ge 2}a_{ij}D_{ij}u=f-\sum_{j=2}^d(a_{j1}+a_{1j})D_{1j}u-a_{11}D_{11}u,
\end{equation*}
which can be regarded as a non-divergence type parabolic equation in $\bR\times \bR^{d-1}$.
Then by applying an interior estimate result without weights for $D^2 u$, where $u$ is, as a function of $x' \in \bR^{d-1}$, a solution of a non-divergence type equation in the whole space $\bR^{d-1}$ (see, for instance, \cite{Kry97}), we bound the mean oscillations of $D_{x^\prime}^2u$ in the $x^\prime$ variables by the maximal functions of $f$, $DD_1 u$ and $D_{x^\prime}^2 u$.

As an application of our results, in a forthcoming paper we are going to study non-divergence form SPDEs in weighted or unweighted Sobolev spaces with the Neumann boundary condition. A particular case is the solvability of SPDEs in the form
\begin{equation*}
du=(a_{ij}D_{ij} u+b_iD_i u+cu-\lambda u+f)\,dt+g_k\, dw^k_t\quad \text{in}\,\,\, (-\infty,T) \times\bR^d_+,
\end{equation*}
where $w^k_t$ are independent one-dimensional Wiener processes, $a_{ij}$, $b_i$, and $c$ satisfy the same conditions as in the current paper, and $f$, $Dg_k$,  $g_k\in \bL_{p,\theta}(-\infty, T)$; see the definition of the $\bL_{p,\theta}$ space at the beginning of Section 2. We note that SPDEs in weighted Sobolev spaces with the {\em Dirichlet boundary condition} have been studied extensively in the past fifteen years. We refer the reader to \cite{KrLo99a, KrLo99b, Kry09, MR3174219} and the references therein.

This paper is organized as follows. In the next section, we introduce some notation and state our main results. In Section 3, we obtain the mean oscillation estimates for $D_{x^\prime}^2u$ and $DD_1u$ separately for a parabolic equation with simple coefficients. In Section 4, we prove our main theorem (Theorem \ref{main}).

\section{preliminaries and main results}
Throughout the paper we use, for example, the following Einstein summation convention: $a_{ij}D_{ij}u=\sum_{i,j}a_{ij}D_{ij}u$.
We introduce some notation used in the paper. As hinted in the introduction, a point in $\mathbb{R}^d$ is denoted by $x=(x_1,\cdots,x_d)$, and also by $x=(x_1,x^\prime)$, where $x^\prime \in \mathbb{R}^{d-1}$. Recall $\bR^{d}_+=\{x:x_1>0\}$. In the parabolic case a point in $\mathbb{R}^{d+1}=\mathbb{R}\times\mathbb{R}^d$
is denoted by $X=(t,x)$.
Set $\bR^{d+1}_+=\{(t,x): x_1> 0\}$.
For $r>0$, let $B^\prime_r(x^\prime)$ be the open ball in $\bR^{d-1}$ of radius $r$ with center $x^\prime$. Denote
\begin{align*}
&B_r(x)=B_r(x_1,x^\prime)=(x_1-r,x_1+r)\times B^\prime_r(x^\prime),\\
& Q_r(t,x)=(t-r^{2},t)\times B_r(x),\\
&B_r^+(x)=B_r(x)\cap \bR_+^{d},\quad Q_r^+(t,x)=Q_r(t,x)\cap \bR^{d+1}_+.
\end{align*}
For $a \in \bR$, we use $Q_r(a)$ to denote $$Q_r(0,a,0)=(-r^2,0)\times(a-r,a+r)\times B_r^\prime(0),$$
and $Q_r=Q_r(0)$. Similarly, we define $Q_r^+(a)$ and $Q_r^+$.
By $a^+$ we mean $\max \{ a, 0 \}$.

Throughout the paper, we assume that the leading coefficients $a_{ij}$ are bounded,  measurable, and satisfy the ellipticity condition:
\begin{equation*}
a_{ij}\eta_i\eta_j\ge \delta|\eta|^2,\quad |a_{ij}|\le 1/\delta
\end{equation*}
for any $\eta\in \bR^d$, where $\delta>0$ is a constant.

To introduce the function spaces used in this paper, we first recall
the weighted Sobolev spaces $H_{p,\theta}^\gamma$ introduced in \cite{Kry99}. If $\gamma$ is a non-negative integer
\begin{equation*}
H_{p,\theta}^\gamma=H_{p,\theta}^\gamma(\bR^{d}_+)=\{u: x_1^{|\alpha|}D^\alpha u\in L_{p,\theta}(\bR^{d}_+) \quad \forall\, \alpha: 0\le |\alpha|\le \gamma\},
\end{equation*}
where $L_{p,\theta}(\bR^d_+) (= L_{p,\theta})$ is a Lebesgue space with the measure $\mu_d(dx) =x_1^{\theta-d}\,dx$.
For a general real number $\gamma$, $H_{p,\theta}^\gamma$ is defined as follows. Take and fix a nonnegative function $\zeta\in C_0^\infty(0,\infty)$ such that
\begin{equation*}
\sum_{n=-\infty}^\infty\zeta^p(e^{x_1-n})\ge1
\end{equation*}
for all $x_1\in \bR$. For any $\gamma, \theta\in \bR$ and $p\in (1,\infty)$, let $H_{p,\theta}^\gamma$ be the set of all functions $u$ on $\bR^d_+$ such that
$$\|u\|_{\gamma,p,\theta}^p=\sum_{n=-\infty}^\infty e^{n\theta}\|u(e^n\cdot)\zeta(x_1)\|^p_{\gamma,p}<\infty,$$
where $\|\cdot\|_{\gamma,p}$ is the norm in the Bessel potential space $H_{p}^\gamma(\bR^d)$.
For any $a\in \bR$, let $M^\alpha$ be the operator of multiplication by $(x_1)^\alpha$ and $M:=M^1$. We write $u\in M^\alpha H_{p.\theta}^\gamma$ if $M^{-\alpha}u\in H^\gamma_{p,\theta}$.
We set
$$\bH_{p,\theta}^\gamma(S,T)=L_p((S,T), H_{p,\theta}^\gamma),\quad \bL_{p,\theta}(S,T)=L_p((S,T), L_{p,\theta}),$$
where $-\infty\le S<T\le\infty$.

Our solution spaces are defined as follows. For the elliptic case, we set
$$
W_{p,\theta}^2(\bR^d_+) = \{ u : u, Du, D^2u \in L_{p,\theta}(\bR^d_+)\}.
$$
For the parabolic case,
\begin{align*}
W_{p,\theta}^{1,2}(S, T)
=\{u: u, Du, D^2u, u_t  \in \bL_{p, \theta}(S,T)\},
\end{align*}
where $-\infty \le S < T \le \infty$.

We also use the following H\"older spaces. For a function $f$ on $\cD\subset \bR^{d+1}$, define
\begin{equation*}
[f]_{a,b,\cD}=:\sup\Big\{\frac{|f(t,x)-f(s,y)|}{|t-s|^{a}+|x-y|^b}: (t,x), (s,y)\in \cD, (t,x)\neq(s,y)\Big\},
\end{equation*}
where $a,b\in (0,1]$. For $a\in (0,1]$, we set
$$\|f\|_{\frac{a}{2},a, \cD}=\|f\|_{\infty,D}+[f]_{\frac{a}{2},a.\cD}.$$
The space corresponding to $\|\cdot\|_{a/2,a,\cD}$ is denoted by $C^{a/2,a}(\cD)$.

Throughout the paper, we use the weighted measures:
\begin{equation*}
\mu_d(dx):=(x_1)^{\theta-d}\,dx,\quad\mu(dx\,dt):=(x_1)^{\theta-d}\,dx\,dt,
\end{equation*}
where $\theta\in(d-1,d-1+p)$ and $p\in (1,\infty)$.

Now we state our regularity assumption on the leading coefficients. For a function $g$ on $\bR^{d+1}_+$, denote
\begin{equation*}
[g(t,\cdot)]_{B^+_r(x)}=\dashint_{B^+_{r}(x)}|g(t,y)-\dashint_{B^+_r(x)}g(t,z)\,\mu_d(dz)|\,\mu_d(dy).
\end{equation*}
Then we define the mean oscillation of $g$ in $Q^+_r(s,y)$ with respect to $x$ as
\begin{equation*}
\text{osc}_x(g, Q^+_r(s,y))=\dashint_{s-r^2}^{\,\,s}[g(\tau, \cdot)]_{B^+_r(y)}d\tau,
\end{equation*}
and denote
\begin{equation}
							\label{eq0718_1}
g_R^\#=\sup_{(s,y)\in \bR^{d+1}_+}\sup_{r\le R}\,\text{osc}_x(g, Q^+_r(s,y)).
\end{equation}
Using the notation above with $a_{ij}$ in place of $g$, we state the following regularity assumption on $a_{ij}$ with a sufficiently small parameter $\rho>0$ to be specified later.

\begin{assumption}[$\rho$]
                            \label{assumption}
There exists a positive constant $R_0$ such that
$$A_{R_0}^\#:=\sup_{i,j} \, (a_{ij})_{R_0}^\#\le \rho.$$
\end{assumption}

Note that under this assumption, the coefficients $a^{ij}$ may not have any regularity with respect to $t$.

\begin{remark}
							\label{rem0717_1}
While we have a fixed size of the modulus of regularity $R_0$ above, in \cite{DK14} the size of the modulus is proportional to the distance from the boundary to the location where the mean oscillations of $a_{ij}$ are measured. To express this, one can replace $R$ in \eqref{eq0718_1} by $y_1 R$.
This means, in particular, that the coefficients $a_{ij}$ in \cite{DK14} are allowed to be much rougher {\em near the boundary} than those in this paper.
\end{remark}

For lower-order terms, we assume that the coefficients $b_i$ and $c$ are only measurable (without any regularity assumptions) and bounded so that $$|b_i|,\, |c|\le K$$ for some constant $K>0$.

The following theorems are our main results, the first of which is the unique solvability of parabolic equations.

\begin{theorem}\label{main}
Let $T\in (-\infty,\infty]$, $1<p<\infty$, and $\theta \in (d-1,d-1+p)$ be constants. Then there exist constants $\rho=\rho(d,\delta,p,\theta)>0$ and $\lambda_0=\lambda_0(d,\delta,p,\theta,K,R_0)\ge 0$ such that under Assumption \ref{assumption} ($\rho$)
the following assertions hold.

(i) Suppose that $u\in W_{p,\theta}^{1,2}(-\infty,T)$ satisfies
\begin{equation}
-u_t+a_{ij}D_{ij}u+b_iD_iu+cu-\lambda u=f\label{eq 4.081}
\end{equation}
in $(-\infty,T)\times \bR_+^d$ with the Neumann boundary condition $D_1u=0$ on $\{x_1=0\}$, where $f\in \bL_{p,\theta}(-\infty, T)$. Then we have
\begin{equation}
\|u_t\|_{p,\theta}+\|M^{-1}D_1u\|_{p,\theta}+\|D^2 u\|_{p,\theta}+\sqrt{\lambda}\|Du\|_{p,\theta}+\lambda\|u\|_{p,\theta}\le C_0\|f\|_{p,\theta}\label{main estimate}
\end{equation}
provided that $\lambda\ge \lambda_0$,
where $\|\cdot\|_{p,\theta}=\|\cdot\|_{\bL_{p,\theta}(-\infty,T)}$ and $C_0=C_0(d, \delta, p,\theta)$. In particular, when $b_i=c=0$ and $a_{ij}=a_{ij}(t)$, we can take $\lambda_0=0$.

(ii) For any $\lambda>\lambda_0$ and $f\in \bL_{p,\theta}(-\infty,T)$, there is a unique solution $u\in W_{p,\theta}^{1,2}(-\infty,T)$ to the equation \eqref{eq 4.081}.
\end{theorem}

We now present our results for elliptic equations, where the coefficients are independent of $t$.
Since Assumption \ref{assumption} ($\rho$) does not concern the regularity of coefficients in $t$, we still require the coefficients to satisfy Assumption \ref{assumption} ($\rho$).

By adapting, for example, the proof of \cite[Theorem 2.6]{Kry97} to the results above for parabolic equations, i.e., by regarding elliptic equations as steady state parabolic equations, we obtain the following theorem for elliptic equations.

\begin{theorem}
Let $1<p<\infty$ and $\theta \in (d-1,d-1+p)$ be constants. Then there exist constants $\rho = \rho(d,\delta,p,\theta)>0$ and $\lambda_0 = \lambda_0(d,\delta,p,\theta,K,R_0)\ge 0$ such that under Assumption \ref{assumption} ($\rho$)
the following assertions hold.

(i) Suppose that $u\in W_{p,\theta}^2(\bR^d_+)$ satisfies
\begin{equation}
a_{ij}D_{ij}u+b_iD_iu+cu-\lambda u=f\label{eq 4.081_e}
\end{equation}
in $\bR_+^d$ with the Neumann boundary condition $D_1u=0$ on $\{x_1=0\}$, where $f\in L_{p,\theta}(\bR^d_+)$. Then we have
\begin{equation*}
\|M^{-1}D_1u\|_{p,\theta}+\|D^2 u\|_{p,\theta}+\sqrt{\lambda}\|Du\|_{p,\theta}+\lambda\|u\|_{p,\theta}\le C_0\|f\|_{p,\theta}
\end{equation*}
provided that $\lambda\ge \lambda_0$,
where $\|\cdot\|_{p,\theta}=\|\cdot\|_{L_{p,\theta}(\bR^d_+)}$ and $C_0=C_0(d, \delta, p,\theta)$. In particular, when $b_i=c=0$ and $a_{ij}$ are constant, we can take $\lambda_0=0$.

(ii) For any $\lambda>\lambda_0$ and $f\in L_{p,\theta}(\bR^d_+)$, there is a unique solution $u\in W_{p,\theta}^2(\bR^d_+)$ to the equation \eqref{eq 4.081_e}.
\end{theorem}

\section{equations with coefficients independent of $x$}
\label{sec3}

In this section, we deal with equations in the form
\begin{align}
&u_t-a_{ij}(t)D_{ij}u=f\quad \text{in}\quad \bR^{d+1}_+,\label{eq1}\\
&D_1u=0\quad \text{on}\quad \{x_1=0\}\label{bc1}.
\end{align}
Note that now the coefficients $a_{ij}$  depend only on $t$. Let us state several technical lemmas.
The first one is Hardy's inequality, which can be found in \cite{Kuf85}.
\begin{lemma}\label{hardy}
Let $p\in (1,\infty)$, $\theta\in(d-1, d-1+p)$, and $v\in C_{\text{loc}}^\infty({\overline{\bR^{d+1}_+}})$  such that $v=0$ on $\{x_1=0\}$.
Then
\begin{equation*}
\int_{B_r^+}|v|^py_1^{\theta-d-p}\,dy\le C\int_{B_r^+}|D_1 v|^{p}y_1^{\theta-d}\,dy
\end{equation*}
for any $r\in (0,\infty]$, where $C=C(d, p,\theta)$.
\end{lemma}
We summarize Lemmas 4.2 and 4.3 in \cite{DK14} as the following results, which were proved by localizing the results in \cite{Kimdo07, Dong11} and using the Sobolev embedding theorem.

\begin{lemma}\label{lemma 3.2}
Let $p\in (1,\infty)$ and $\alpha \in (0,1)$ be constants. Assume that $u\in C_{\text{loc}}^\infty(\overline{\bR^{d+1}_+})$ satisfies \eqref{eq1} in $Q_2^+$ with $f=0$, and $u=0$ on $\{x_1=0\}$. Then we have
\begin{equation*}
\|Du\|_{C^{\alpha/2, \alpha}(Q_1^+)}\le C\|u\|_{L_p(Q_2^+)},
\end{equation*}
where $C=C(d,\delta,p,\alpha)$. If $Q_1$ and $Q_2$ replace $Q_1^+$ and $Q_2^+$, respectively, then
$$
\|Du\|_{C^{\alpha/2, \alpha}(Q_1)}\le C\|Du\|_{L_p(Q_2)}.
$$
\end{lemma}

For a domain $\cD\subset \bR^{d+1}_+$, we denote $(u)_{\cD}$ to be the average of $u$ in $\cD$ with respect to the measure $\mu(dx\,dt)=x_1^{\theta-d}\,dx\,dt$. Precisely,
\begin{equation*}
(u)_{\cD}=\frac{1}{\mu(\cD)}\int_{\cD}u(t,x)\mu(dx\,dt)=\dashint_{\cD}u(t,x)\mu(dx\,dt),
\end{equation*}
where
$$
\mu(\cD)=\int_{\cD}\mu(dx\,dt).
$$

Using Lemma \ref{lemma 3.2}, we obtain the following mean oscillation type estimate.
\begin{corollary}\label{cor 3.2}
Let $p\in (1,\infty)$, $\alpha\in (0,1)$, $y_1\in[0,1]$, and $\theta\in (d-1,d-1+p)$ be constants. Assume that $u\in C_{\text{loc}}^\infty(\overline{\bR^{d+1}_+})$ satisfies
\begin{equation*}
u_t-a_{ij}(t)D_{ij}u=0\quad \text{in}\quad Q_{4}^+
\end{equation*}
with the Dirichlet boundary condition $u=0$ on $\{x_1=0\}$. Then there exists a constant $C=C(d,\delta,p,\theta,\alpha)$ such that for any $r<1$,
\begin{equation*}
(|Du- (Du)_{Q^+_r(y_1)}|^p)_{Q^+_r(y_1)}\le Cr^{\alpha p}\int_{Q_4^+}|D_1 u|^p\,\mu(dx\,dt).
\end{equation*}
\end{corollary}
\begin{proof}
From Lemma \ref{lemma 3.2}, we obtain
\begin{equation*}
(|Du - (D u)_{Q^+_r(y_1)}|^p)_{Q^+_r(y_1)}\le r^{\alpha p}[Du]^p_{\frac{\alpha}{2},\alpha,Q_2^+} \le Cr^{\alpha p}\| u\|^p_{L_p(Q_4^+)}.
\end{equation*}
Since $\theta-d-p<0$ and by Lemma \ref{hardy}, we have
$$ \|u\|^p_{L_p(Q_4^+)}\le C\int_{Q_4^+}|u|^px_1^{\theta-d-p}\,dx\,dt\le C\int_{Q_4^+}|D_1 u|^p\,\mu(dx\,dt).$$
Combining the two inequalities above, we prove the corollary.
\end{proof}

Before we state the next theorem, we introduce a function space.
For $\lambda \ge 0$, we denote $u\in \cH_{p,\theta}^{1,\lambda}(-\infty,T)$ if
$$\sqrt{\lambda}u,\,\,M^{-1}u,\,\,Du\in \bL_{p,\theta}(-\infty,T),$$
and $u_t\in M^{-1}\bH^{-1}_{p,\theta}(-\infty,T)+\sqrt{\lambda}\bL_{p,\theta}(-\infty,T)$.
Let us write $\cH_{p,\theta}^1(-\infty,T)$ if $\lambda = 0$.
By \cite[Remark 5.3]{Kry99}, one can find $f, g=(g_1,\ldots,g_d)\in \bL_{p,\theta}(-\infty,T)$ such that
$$u_t=D_ig_i+\sqrt{\lambda}f$$
in $(-\infty,T)\times \bR^d_+$. We set
\begin{align*}
&\|u\|_{\cH_{p,\theta}^{1,\lambda}(-\infty,T)}=\inf\Big\{\sqrt{\lambda}\|u\|_{p,\theta}+
\|M^{-1}u\|_{p,\theta}+\|Du\|_{p,\theta}\\
&\qquad \quad +\|g\|_{p,\theta}+\|f\|_{p,\theta}:u_t=D_ig_i+\sqrt{\lambda}f\Big\}.
\end{align*}

Now we state a special case ($a^{ij} = a^{ij}(t)$) of \cite[Theorem 3.9]{DK14}, where $a^{ij}$ are allowed to be merely measurable in $(t,x_1)$ except that $a^{11}=a^{11}(t)$ or $a^{11}=a^{11}(x_1)$.
Note that in the theorem below there is no specification of the boundary condition, but functions in the solution space $\cH_{p,\theta}^{1,\lambda}(-\infty,T)$ necessarily satisfies $u=0$ on the boundary. Hence the theorem is about the Dirichlet boundary value problem for divergence type equations in weighted Sobolev spaces.

\begin{theorem}\label{theorem 3.8}
Let $T\in (-\infty, \infty]$, $\lambda\ge 0$, $p\in (1,\infty)$, $\theta\in (d-1,d-1+p)$, and $u\in \cH_{p,\theta}^{1,\lambda}(-\infty, T)$ satisfy
\begin{equation}
u_t-D_i(a_{ij}D_j u)+\lambda u=D_ig_i+f \label{eq 3.18}
\end{equation}
in $(-\infty,T)\times \bR_+^{d}$, where $g=(g_1,g_2,\ldots,g_d)$, $g,f\in \mathbb{L}_{p,\theta}(-\infty,T)$, and $f\equiv 0$ if $\lambda=0$. Then we have
\begin{equation*}
\sqrt{\lambda}\|u\|_{p,\theta}+\|M^{-1}u\|_{p,\theta}+\|Du\|_{p,\theta}\le C\|g\|_{p,\theta}+\frac{C}{\sqrt{\lambda}}\|f\|_{p,\theta},
\end{equation*}
where $C=C(d,\delta,p,\theta)$.
Moreover, for any $g,f\in \mathbb{L}_{p,\theta}(-\infty,T)$ such that $f\equiv 0$ if $\lambda=0$, there exists a unique solution $u\in \cH_{p,\theta}^{1,\lambda}(-\infty,T)$ to the equation \eqref{eq 3.18}.
\end{theorem}

Next we consider \eqref{eq1}-\eqref{bc1} with non-vanishing right-hand side.
\begin{lemma}\label{normal simple}
Let $p\in (1,\infty)$, $\theta\in (d-1,d-1+p)$, $\alpha\in (0,1)$, $r>0, \kappa \ge 32$, and $y_1\ge 0$. Assume that $f\in \mathbb{L}_{p,\theta}(Q_{\kappa r}^+(y_1))$ and $u\in C_{\text{loc}}^\infty(\overline{\bR^{d+1}_+})$ is a solution of
$$u_t-a_{ij}D_{ij}u=f$$
in $Q^+_{\kappa r}(y_1)$ with the Neumann boundary condition $D_1u=0$ on $\{x_1=0\}$. Then we have
\begin{align*}
&(|DD_1 u-(DD_1u)_{Q_r^+(y_1)}|^p)_{Q_r^+(y_1)}\\
&\le C\kappa^{-\alpha p}(|DD_1u|^p)_{Q_{\kappa r}^+(y_1)}
+C\kappa^{d+\theta+2}(|f|^p)_{Q_{\kappa r}^+(y_1)},
\end{align*}
where $C=C(d,\delta,p,\theta ,\alpha)$.
\end{lemma}
\begin{proof}
Denote $v=D_1u$.  Then $v$ satisfies
\begin{equation*}
v_t-D_i(a_{ij}D_jv)=D_1 f
\end{equation*}
in $Q^+_{\kappa r}(y_1)$ with the Dirichlet boundary condition $v=0$ on $\{x_1=0\}$.
By a scaling argument, it is sufficient to set $\kappa r=8$. We consider two cases.

{\em Case 1:} $y_1\in [0,1]$. Since $r=8/\kappa\le 1/4$, we have
\begin{equation*}
Q_r^+(y_1)\subset Q_2^+\subset Q_4^+\subset Q_{\kappa r}^+(y_1).
\end{equation*}
Let $\eta$ be a smooth function with support in $(-(\kappa r)^2, (\kappa r)^2) \times B_{\kappa r}(y_1,0)$ and $\eta=1$ in $Q_4$. By Theorem \ref{theorem 3.8}, there exists a unique solution $w\in \cH_{p,\theta}^{1,0}(-\infty,0) = \cH^1_{p,\theta}(-\infty,0)$ to the equation
\begin{equation*}
w_t-D_i(a_{ij}D_jw)=D_1(f\eta)
\end{equation*}
in $(-\infty,0)\times \bR_+^d$, satisfying
\begin{equation*}
\|Dw\|_{p,\theta}\le C\|f\eta\|_{p,\theta}.
\end{equation*}
Due to the definition of $\eta$, this implies
\begin{equation}\label{eq 220.1}
\|Dw\|_{p,\theta}^p\le C \mu ( Q_{\kappa r}^+(y_1) ) \left( |f|^p \right)_{Q_{\kappa r}^+(y_1)}.
\end{equation}

By a standard mollification argument (see, for instance, \cite[Theorem 4.7]{HZ14}), we may assume that $w$ is smooth.
Let $\hat{u}=v-w$, which is also smooth, and satisfies  $\hat{u} = 0$ on $\{x_1 = 0 \}$ and
$$\hat{u}_t-D_i(a_{ij}D_j\hat{u})=0$$
in $Q_4^+$. By Corollary \ref{cor 3.2}, we have
\begin{equation}
(|D\hat{u}- (D\hat{u})_{Q^+_r(y_1)}|^p)_{Q^+_r(y_1)}\le Cr^{\alpha p}\int_{Q^+_4}|D_1 \hat{u}|^p\,\mu(dx\,dt).\label{eq 220.2}
\end{equation}
Combining \eqref{eq 220.1},  \eqref{eq 220.2}, and the triangle inequality, we reach
\begin{align*}
&(|DD_1 u-(DD_1u)_{Q_r^+(y_1)}|^p)_{Q_r^+(y_1)}\\
&\le C(|D\hat{ u}-(D\hat{u})_{Q_r^+(y_1)}|^p)_{Q_r^+(y_1)}+C(|Dw|^p)_{Q_r^+(y_1)}\\
&\le Cr^{\alpha p}(|D\hat{u}|^p)_{Q^+_4}+C\mu(Q_r^+(y_1))^{-1}\mu(Q_{\kappa r}^+(y_1))(|f|^p)_{Q_{\kappa r}^+(y_1)}\\
&\le Cr^{\alpha p}(|DD_1 u|^p)_{Q^+_4}+ C\mu(Q_r^+(y_1))^{-1}\mu(Q_{\kappa r}^+(y_1)) (|f|^p)_{Q_{\kappa r}^+(y_1)}.
\end{align*}
Bearing in mind that $\mu(Q_r^+(y_1))\ge Cr^{\sigma+2}$, where $\sigma=\max\{d,\theta\}$ and $r=8/\kappa$, we prove the lemma for Case 1.

{\em Case 2:} $y_1>1$. This is essentially an interior case.  Since $r=8/\kappa\le 1/4$, we have
$$
Q_r^+(y_1)=Q_r(y_1)\subset Q_{1/4}(y_1)\subset Q_{1/2}(y_1)\subset Q_{\kappa r}^+(y_1).
$$
As in Case 1, we take a smooth function $\eta$  with support in $(-(\kappa r)^2, (\kappa r)^2) \times B_{\kappa r}(y_1,0)$ and $\eta=1$ on $Q_{1/2}(y_1)$. Then by Theorem \ref{theorem 3.8}, there exists a unique solution  $w\in\cH^1_{p,\theta}(-\infty,0)$ of the  equation
$$w_t-D_i(a_{ij}D_jw)=D_1(f\eta)
$$
in $(-\infty,0)\times \bR_+^d$, satisfying
$$
\|Dw\|_{p,\theta}\le C\|f\eta\|_{p,\theta}.
$$
Then we get
\begin{equation}
							\label{eq0703_1}
\|Dw\|_{p,\theta}^p \le C \mu ( Q_{\kappa r}^+(y_1) ) \left( |f|^p \right)_{Q_{\kappa r}^+(y_1)}.
\end{equation}

For the same reason as before, we may assume that $w$ is smooth and so is $\hat{u}=D_1u-w$.
It is easily seen that
\begin{equation*}
\hat{u}_t-a_{ij}D_{ij}\hat{u}=0\quad \text{in}\quad Q_{1/2}(y_1).
\end{equation*}
By Lemma \ref{lemma 3.2}, we have
\begin{equation*}
\|D\hat{u}\|_{C^{\alpha/2,\alpha}(Q_{1/4}(y_1))}\le C\|D\hat{u}\|_{L_p(Q_{1/2}(y_1))}.
\end{equation*}
From this and \eqref{eq0703_1},
\begin{align}\nonumber
&(|DD_1 u-(DD_1u)_{Q_r(y_1)}|^p)_{Q_r(y_1)}\\\nonumber
&\le C(|D\hat{u}-(D\hat{u})_{Q_r(y_1)}|^p)_{Q_r(y_1)}+C(|Dw|^p)_{Q_r(y_1)}\\\nonumber
&\le Cr^{\alpha p}\|D\hat{u}\|^p_{C^{\alpha/2,\alpha}(Q_{1/4}(y_1))}+C\mu(Q_r(y_1))^{-1}\mu(Q_{\kappa r}^+(y_1))(|f|^p)_{Q_{\kappa r}^+(y_1)}\\
&\le Cr^{\alpha p}\|D\hat{u}\|^p_{L_p(Q_{1/2}(y_1))}+C\mu(Q_r(y_1))^{-1}\mu(Q_{\kappa r}^+(y_1))(|f|^p)_{Q_{\kappa r}^+(y_1)}.\label{eq 220.3}
\end{align}
Since for any $x_1\in (y_1-1/2, y_1+1/2)$,
$$
C_1\le \frac{x_1^{\theta-d}}{\mu(Q_{1/2}(y_1))}\le C_2,
$$
where $C_{1,2}=C_{1,2}(d,\theta)$, by \eqref{eq 220.3}, \eqref{eq0703_1}, and the triangle inequality, it holds that
\begin{align*}
&(|DD_1 u-(DD_1u)_{Q_{r}(y_1)}|^p)_{Q_r(y_1)}\\
&\le Cr^{\alpha p} (|DD_1 u|^p)_{Q_{1/2}(y_1)}
+ C \mu(Q_r(y_1))^{-1}\mu(Q_{\kappa r}^+(y_1)) (|f|^p)_{Q_{\kappa r}^+(y_1)}.
\end{align*}
Taking into account that $\mu(Q_r(y_1))\ge Nr^{d+2}\mu(Q_8^+(y_1))$, we prove the lemma for the second case.
\end{proof}


It remains to estimate $D_{ij} u$ with $i,j>1$.
Let us first state a  Poincar\'e inequality in weighted $L_p$ spaces.

\begin{lemma}\label{lemma 4.1}
Let $a\in \overline{\bR^+}$, $\alpha\in (-1,\infty)$, $r\in (0,\infty)$, $p\in [1,\infty)$, $\nu(dx)=(x_1)^\alpha\, dx$, $x\in \bR^{d}_+$, and $u\in C^\infty_{\text{loc}}(\overline{\bR^d_+})$. Then
\begin{equation}
							\label{eq0703_2}
\begin{aligned}
&\int_{B^+_r(a)}\int_{B^+_r(a)}|u(x)-u(y)|^p\,\nu(dx)\,\nu(dy)\\
&\le C r^p \nu(B^+_r(a))\int_{B^+_r(a)}|Du|^p\,\nu(dx),
\end{aligned}
\end{equation}
where $C = C(d,p,\alpha)$.
\end{lemma}

\begin{proof}
When $\alpha \in [0,\infty)$, the inequality is proved in \cite[Lemma 4.1]{KimLee13} with a missing constant depending on $d$.
For the sake of completeness, we here present a proof when $\alpha\in (-1,\infty)$.
Since the weight is with respect to $x_1$, we only prove \eqref{eq0703_2} for $d = 1$. In fact, to prove \eqref{eq0703_2} for $d>1$ we just need to combine the case when $d=1$ with the unweighted Poincar\'{e} inequality.

Due to scaling, it suffices to prove \eqref{eq0703_2} with $r=1$. We further assume $a\in (0,2)$. Indeed, if $a\ge 2$, the inequality \eqref{eq0703_2} is equivalent to the usual Poincar\'e inequality without weights.
For each $x, y \in ((a-1)^+,a+1)$,
by H\"{o}lder's inequality,
\begin{align*}
&|u(x)-u(y)|^p = \Big|\int_{\min\{x,y\}}^{\max\{x,y\}}u'(z)\,dz\Big|^p\\
&\le C \left(y^{1-\alpha q/p}+ x^{1-\alpha q/p}\right)^{p/q} \int_{\min\{x,y\}}^{\max\{x,y\}} |u'(z)|^p z^\alpha \, dz\\
&\le C \left(y^{1-\alpha q/p}+ x^{1-\alpha q/p}\right)^{p/q} \int_{(a-1)^+}^{a+1} |u'(z)|^p z^\alpha \, dz,
\end{align*}
where $1/p+1/q=1$ and $C = C(p,\alpha)$.
Then, to conclude \eqref{eq0703_2}, we integrate the above inequalities with respect to $\nu(dx)$ and $\nu(dy)$, and use the fact that $a \in (0,2)$.
\end{proof}

In the sequel, we denote the standard parabolic cylinder in $\bR^{d+1}$ as
$$\hat{Q}^{d+1}_r(X)=(t-r^2,t)\times \hat{B}_r(x),$$ where $\hat{B}_r(x)$ is the Euclidean ball in $\bR^{d}$ with radius $r$ and center $x$.
For a function $f$ on $\bR^{d+1}$, we define the average of $f$ in $\hat{Q}^{d+1}_r(X)$ without weight as
$$
\langle f \rangle_{\hat{Q}^{d+1}_r(X)}
=\frac{1}{|\hat{Q}^{d+1}_r(X)|}\int_{\hat{Q}^{d+1}_r(X)}f(t,x)\,dx\,dt.
$$

\begin{theorem}\label{kry}
Let $p\in (1,\infty)$ and $u\in C_{\text{loc}}^\infty(\bR^{d+1})$ satisfy
\begin{equation*}
- u_t + a_{ij}(t)D_{ij}u=f\quad \text{in}\quad \bR^{d+1}.
\end{equation*}
Then there exists a constant $C=C(d,\delta,p)$ such that for any $\kappa\ge 4, r>0$, we have
$$
\left\langle |D^2 u- \langle D^2 u \rangle_{\hat{Q}^{d+1}_r}|^p \right\rangle_{\hat{Q}^{d+1}_r}\le C\kappa^{-p}\langle |D^2 u|^p \rangle_{\hat{Q}^{d+1}_{\kappa r}}+C \kappa^{d+2}\langle |f|^p \rangle_{\hat{Q}^{d+1}_{\kappa r}}.
$$
\end{theorem}
\begin{proof}
See \cite[Theorem 5.1]{Kry07JFA}. 
\end{proof}

To estimate $D_{x^\prime}^2u$, we introduce the following notation.
For $\mathfrak{y}_1 \ge 0$, set $B^{1+}_r(\mathfrak{y}_1)=((\mathfrak{y}_1-r)^+,\mathfrak{y}_1+r)\subset \bR$, $\mu_1(dx_1)=x_1^{\theta-d}\,dx_1$, and
$$
\mu_1(B_r^{1+}(\mathfrak{y}_1))
=\int_{(\mathfrak{y}_1-r)^+}^{\mathfrak{y}_1+r}x_1^{\theta-d}\,dx_1.
$$
\begin{lemma}\label{key lemma}
Let $p\in(1,\infty), \theta\in (d-1,d-1+p), r>0, \kappa\ge 32$, and $\mathfrak{y}_1\ge 0$. Assume that $f\in \bL_{p,\theta}(Q_{\kappa r}^+(\mathfrak{y}_1))$ and  $u\in C_{\text{loc}}^\infty(\overline{\bR^{d+1}_+})$ is a solution of
\begin{equation}
-u_t+a_{ij}D_{ij}u=f \quad\label{eq 5.222}
\end{equation}
in $Q_{\kappa r}^+(\mathfrak{y}_1)$ with $D_1 u=0$ on $\{x_1=0\}$. Then there exist constants $C=C(d,\delta,p, \theta)$ and $\alpha=\alpha(d, p,\theta)>0$ such that
\begin{align*}\nonumber
&(|D_{x^\prime}^2 u-(D_{x^\prime}^2 u)_{Q_r^+(\mathfrak{y}_1)}|^p)_{Q_r^+(\mathfrak{y}_1)}\\\nonumber
&\le  C \kappa^{-\alpha}(|D_{x^\prime}^2 u|^p)_{Q^+_{\kappa r}(\mathfrak{y}_1)} +C\kappa^{d+\theta+2}\big((|f|^p)_{Q^+_{\kappa r}(\mathfrak{y}_1)}+(|DD_1 u|^p)_{Q^+_{\kappa r}(\mathfrak{y}_1)}\big).
\end{align*}
\end{lemma}
\begin{proof}
By scaling, we may assume that $\kappa r=8$. In this case $r\le 1/4$ because $\kappa \ge 32$.
Let $\eta=\eta(t,x')\in C_0^\infty(\bR\times \bR^{d-1})$ with a unit integral such that $\text{supp}(\eta)\subset Q^{d}_r$ and
\begin{equation}
							\label{eq0708_1}
|D_{x'} \eta | \le C(d) \, r^{-d-2}
\quad
\text{on}
\,\,\, Q_r^d,
\end{equation}
where
$$
Q_r^{d}:={(-r^2,0)}\times B_r^{d-1}:={(-r^2,0)}\times\{|x^\prime|\le r\}.
$$
We consider two cases.

{\em Case (i):} $\mathfrak{y}_1\in [0,1]$.  By H\"older's inequality and the triangle inequality, it is easily seen that
\begin{align}\nonumber
&\dashint_{Q^+_r(\mathfrak{y}_1)}|D^2_{x^\prime}u-(D^2_{x^\prime}u)_{Q^+_r(\mathfrak{y}_1)}|^p
\,\mu(dx\,dt)\\
&\le C\dashint_{Q^+_r(\mathfrak{y}_1)}\dashint_{Q^+_r(\mathfrak{y}_1)}|D_{x^\prime}^2u(t,x)-D_{x^\prime}^2u(s,y)|^p
\,\mu(dx\,dt)\,\mu(dy\,ds).\label{eq 4.201}
\end{align}
Since, by the triangle inequality,
\begin{align*}
&|D_{x^\prime}^2u(t,x)-D_{x^\prime}^2u(s,y)|^p\\
&\le C\Big|D_{x^\prime}^2u(t,x)-\int_{Q_r^d}\eta(\sigma,z^\prime)D_{x^\prime}^2u(\sigma, x_1, z^\prime)\,dz^\prime\,d\sigma\Big|^p\\
&\quad +C\Big|D_{x^\prime}^2u(s,y)-\int_{Q_r^d}\eta(\sigma,z^\prime)D_{x^\prime}^2u(\sigma, y_1, z^\prime)\,dz^\prime\,d\sigma\Big|^p\\
&\quad +C\Big|\int_{Q_r^d}\eta(\sigma, z^\prime) (D_{x^\prime}^2u(\sigma, x_1,z^\prime)-D_{x^\prime}^2u(\sigma, y_1,z^\prime))\,dz^\prime\,d\sigma\Big|^p,
\end{align*}
the right-hand side of \eqref{eq 4.201} is bounded by $C(I+II)$, where
\begin{align*}
I:&=\dashint_{Q^+_r(\mathfrak{y}_1)}\Big|D^2_{x^\prime}u(t,x_1,x^\prime)-\int_{Q^{d}_r}\eta(s,y^\prime)D^2_{x^\prime}u(s,x_1,y^\prime)\,dy^\prime\,ds\Big|^p
\,\mu(dx\,dt),\\
II:&=\dashint_{B_r^{1+}(\mathfrak{y}_1)}\dashint_{B^{1+}_r(\mathfrak{y}_1)}\Big|\int_{Q_r^d}(D^2_{x^\prime}u(t,x_1, z^\prime)-D^2_{x^\prime}u(t, y_1, z^\prime))\eta(t,z^\prime)\,dz^\prime\,dt\Big|^p\\
&\quad \cdot\mu_1(dx_1)\,\mu_1(dy_1).
\end{align*}

Let us now estimate $I$ and $II$ separately and first consider $II$. By integrating by parts and H\"older's inequality,
\begin{align*}
&\Big|\int_{Q_r^d}(D^2_{x^\prime}u(t,x_1, z^\prime)-D^2_{x^\prime}u(t, y_1, z^\prime))\eta(t,z^\prime)dz^\prime\, dt\Big|^p\\
&=\Big|\int_{Q_r^d}(D_{x^\prime}u(t,x_1, z^\prime)-D_{x^\prime}u(t, y_1, z^\prime))D_{x^\prime}\eta(t,z^\prime)dz^\prime\, dt\Big|^p\\
&\le \int_{Q_r^d}|D_{x^\prime}u(t,x_1, z^\prime)-D_{x^\prime}u(t, y_1, z^\prime)|^pdz^\prime\, dt \cdot (\int_{Q_r^d}|D_{x^\prime}\eta|^qdz^\prime\, dt)^{\frac{p}{q}},
\end{align*}
where $1/p+1/q=1$. From \eqref{eq0708_1}, we have
\begin{equation*}
\Big(\int_{Q_r^d}|D_{x^\prime}\eta|^qd\sigma\,dz^\prime\Big)^{\frac{p}{q}}\le Cr^{-(d+1+p)}.
\end{equation*}
We plug the two inequalities above into $II$ to achieve
\begin{align*}
&II\le\\
& Cr^{-(d+1+p)}\dashint_{B^{1+}_r(\mathfrak{y}_1)}\dashint_{B^{1+}_r(\mathfrak{y}_1)} \int_{Q_r^d}|D_{x^\prime}u(t,x_1, z^\prime)-D_{x^\prime}u(t, y_1, z^\prime)|^pdz^\prime\, dt\\
&\cdot\mu_1(dx_1)\,\mu_1(dy_1).
\end{align*}
Applying Lemma \ref{lemma 4.1} with $d=1$, we get
\begin{equation*}
II\le C\dashint_{Q^+_r(\mathfrak{y}_1)}|D_1D_{x^\prime}u|^p \,\mu(dx\,dt) \le C\kappa^{d+\theta+2}(|DD_1 u|^p)_{Q^+_{\kappa r}(\mathfrak{y}_1)}
\end{equation*}
because $|Q^+_r(\mathfrak{y}_1)|\ge Cr^{d+\theta+2}$.

Next, we estimate $I$, which can be written as
\begin{equation*}
I=\dashint_{Q^+_r(\mathfrak{y}_1)}\Big|\int_{Q^{d}_r}\eta(s,y^\prime)(D_{x^\prime}^2 u(t,x_1,x^\prime)-D_{x^\prime}^2 u(s,x_1,y^\prime))\,dy^\prime\,ds\Big|^p\,\mu(dx\,dt).
\end{equation*}
By H\"older's inequality,
\begin{align*}
I&\le  \dashint_{Q^+_r(\mathfrak{y}_1)}\int_{Q_r^d}|D^2_{x^\prime}u(t,x_1,x^\prime)-D^2_{x^\prime}u(s,x_1,y^\prime)|^p\,dy^\prime\,ds\,
\mu(dx\,dt)\\
&\quad \cdot \left( \int_{Q_r^d}|\eta|^q\,dy^\prime\,ds \right)^{\frac{p}{q}}.
\end{align*}
Since
\begin{equation*}
\left( \int_{Q_r^d}|\eta|^q\,dy^\prime\,ds \right)^{\frac{p}{q}}\le Cr^{-(d+1)},
\end{equation*} we have
\begin{equation}
							\label{eq 3.1}
\begin{aligned}
I \le & C\dashint_{Q^+_r(\mathfrak{y}_1)}\dashint_{Q_r^d}|D^2_{x^\prime}u(t,x_1,x^\prime)-D^2_{x^\prime}u(s,x_1,y^\prime)|^p \, dy^\prime\,ds\,\mu(dx\,dt)
\\
& = C \dashint_{B_r^{1+}(\mathfrak{y_1})} I(x_1) \, x_1^{\theta-d} \,dx_1,
\end{aligned}
\end{equation}
where
$$
I(x_1)=\dashint_{Q_r^{d}}\dashint_{Q_r^d}|D^2_{x^\prime}u(t,x_1,x^\prime)-D^2_{x^\prime}u(s,x_1,y^\prime)|^p\,dy^\prime\,ds\,dx^\prime\, dt.
$$
We now estimate $I(x_1)$ by writing the equation \eqref{eq 5.222} as
\begin{equation*}
-u_t+\sum_{i,j>1}a_{ij}D_{ij}u=-a_{11}D_{11}u-\sum_{j=2}^d(a_{1j}+a_{j1})D_{1j}u+f.
\end{equation*}
Here, for each fixed $x_1 \in (0,\infty)$, we regard $u(t,x_1,x')$ as a solution to the above equation defined in $\bR \times \bR^{d-1}$. Then thanks to Theorem \ref{kry} with $d$ in place of $d+1$ and the triangle inequality,
we get
\begin{align*}
&I(x_1)\le C \dashint_{Q_r^d}|D^2_{x^\prime}u(t,x_1,x^\prime)-(D^2_{x^\prime}u)_{Q_r^d}(x_1)|^p\,dx^\prime \,dt\\
&\le C\kappa^{-p}\langle |D^2_{x^\prime}u|^p\rangle_{Q^d_{\kappa r}}(x_1)+C\kappa^{d+1}\langle |f|^p \rangle_{Q_{\kappa r}^d}(x_1)
+C\kappa^{d+1}\langle |DD_1 u|^p \rangle_{Q_{\kappa r}^d}(x_1),
\end{align*}
where, for fixed $x_1$,  $\langle |g|^p \rangle_{Q_{\kappa r}^d}(x_1)$ is the unweighted average of $g$ with respect to $(t, x^\prime)$ in the $d$ dimensional parabolic cylinder with radius $\kappa r$. We plug the inequality above into \eqref{eq 3.1} to obtain
\begin{align}
I&\le C\dashint_{B_r^{1+}(\mathfrak{y}_1)}\bigg(\kappa^{-p}\langle |D^2_{x^\prime}u|^p \rangle_{Q^d_{\kappa r}}(x_1)+\kappa^{d+1}\langle |f|^p \rangle_{Q_{\kappa r}^d}(x_1)\nonumber\\
&\qquad+\kappa^{d+1} \langle |DD_1 u|^p \rangle_{Q_{\kappa r}^d}(x_1)\bigg) x_1^{\theta-d}\,dx_1.\label{eq 4.141}
\end{align}
Bearing in mind that $\mu_1(B_r^{1+}(\mathfrak{y}_1))\ge Cr^{\zeta+1-d}$, where
$\zeta=\max\{d,\theta\}$, and $r=8/\kappa$, we obtain from \eqref{eq 4.141} that
\begin{align*}
I&\le C\kappa^{-(p+d-\zeta-1)}(|D_{x^\prime}^2u|^p)_{Q^+_{\kappa r}(\mathfrak{y}_1)}\\
&\quad +C\kappa^{\zeta+2} \left( (|f|^p)_{Q^+_{\kappa r}(\mathfrak{y}_1)}+(|DD_1 u|^p)_{Q^+_{\kappa r}(\mathfrak{y}_1)} \right).
\end{align*}
By the definition of $\zeta$, $p+d-\zeta-1>0$.
Combining the estimates of $I$ and $II$, we complete the proof of Case (i) with $\alpha=p+d-\zeta-1$.

{\em Case (ii):} $\mathfrak{y}_1>1$.
Set $v = D_k u$, $k = 1, \ldots,d$, and note that $v$ satisfies the divergence type equation
$$
v_t - D_i(a_{ij} D_j v) = D_k f
$$
in $B_{\kappa r}(\mathfrak{y}_1)$.
Since this is an interior estimate, we do not care about the boundary value of $v$ on $\{x_1 = 0\}$.
Then we repeat the second part of the proof of Lemma \ref{normal simple} with $D_k$ in place of $D_1$.
The lemma is proved.
\end{proof}

\section{proof of theorem \ref{main}}
In this section, we deal with operators with coefficients depending on both $x$ and $t$. We denote $L=a_{ij}D_{ij}$ and assume $p>1$, $\lambda\ge 0$, and $\theta\in (d-1,d-1+p)$.

First we need the following lemma.
\begin{lemma}
							\label{lem0716}
Let $0 < r \le R < \infty$ and $X = (t, x)$, $Y=(s,y) \in \overline{\bR^{d+1}_+}$. Then
\begin{equation}
                                \label{eq9.22}
\mu \left(Q_r^+(X)\right) \le C(d,\theta)\, \mu\left( Q^+_R(Y) \right)
\end{equation}
provided that $Q_r^+(X) \cap Q_R^+(Y) \ne \emptyset$.
\end{lemma}

\begin{proof}
By scaling, without loss of generality we may assume that $R=1$. We then consider two cases depending on $y_1$.

{\em Case 1:} $y_1>10$. In this case, we have
$$
x_1>8,\quad 1<\max \{y_1+1,x_1+r\}/\min\{y_1-1,x_1-r\}<11/7,
$$
which allows us to use the Lebesgue measure in comparing $\mu\left(Q_r^+(X)\right)$ and $\mu\left(Q_1^+(Y)\right)$. Therefore, \eqref{eq9.22} clearly holds since $r\le 1$.

{\em Case 2:} $y_1\in [0,10]$. In this case, we have $x_1\in [0,12)$. It is then easily seen that $\mu(Q_1^+(Y))$ is bounded from below by a constant $c(d,\theta)>0$ and $\mu(Q_r^+(X))\le \mu(Q_1^+(X))$ is bounded from above by a constant $C(d,\theta)>0$. Thus, \eqref{eq9.22} still holds.

The lemma is proved.
\end{proof}

The following two lemmas are mean oscillation estimates for the operator $a_{ij}(t,x) D_{ij}$. We prove them by using the mean oscillation estimates for $a_{ij}(t) D_{ij}$ proved in Section \ref{sec3} combined with a perturbation argument.

\begin{lemma}\label{mean 1}
Let $R>0$, $\kappa \ge 32$, and $\beta, \beta^\prime\in (1, \infty)$ satisfying   $1/\beta+1/\beta^\prime=1$. Suppose that $u\in C_{\text{loc}}^{\infty}(\overline{\bR^{d+1}_+})$ is compactly supported on $\overline{Q^+_R(X_1)}$, where $X_1=(\hat t,\hat{x})\in \overline{\bR^{d+1}_+}$. Moreover, $D_1u=0$ on $\{x_1=0\}$ and $f:=-u_t+Lu$. Then for any $r>0$ and $Y=(s,y)\in \overline{\bR^{d+1}_+}$, we have
\begin{align}\nonumber
&(|DD_1u-(DD_1 u)_{Q_r^+(Y)}|^p)_{Q_r^+(Y)}\le C_0\kappa^{-p/2}(|DD_1 u|^p)_{Q_{\kappa r}^+(Y)}\\
&+C_0\kappa^{d+\theta+2}(|f|^p)_{Q^+_{\kappa r}(Y)}+C_1\kappa^{d+\theta+2}(A_R^\#)^{1/\beta^{\prime}}(|D^2 u|^{\beta p})_{Q_{\kappa r}^+(Y)}^{1/\beta},\label{eq 4.107}
\end{align}
where $C_0=C_0(d,\delta, p,\theta)$ and $C_1=C_1(d,\delta,p,\theta, \beta)$.
\end{lemma}

\begin{proof}
Throughout the proof, we assume $Q_R^+(X_1) \cap Q_{\kappa r}^+(Y) \ne \emptyset$. Otherwise, \eqref{eq 4.107} holds trivially.
Fix a $z\in \overline{\bR^{d}_+}$ and set
$$L_zu=a_{ij}(t,z)D_{ij}u(t,x).$$
Then we have
$$-u_t+L_zu=f+(a_{ij}(t,z)-a_{ij}(t,x))D_{ij}u(t,x):=\hat{f}.$$
It follows from Lemma \ref{normal simple} with $\alpha=1/2$ and a translation of the coordinates that
\begin{align}\nonumber
&(|DD_1u- (DD_1u)_{Q_r^+(Y)}|^p)_{Q_r^+(Y)}\\
&\le C_0\kappa^{-p/2}(|DD_1u|^p)_{Q_{\kappa r}^+(Y)}+C_0\kappa^{d+\theta+2}(|\hat{f}|^p)_{Q_{\kappa r}^+(Y)},\label{eq 4.104}
\end{align}
where $C_0=C_0(d,\delta,p,\theta)$. By the definition of $\hat{f}$, the triangle inequality, and the fact that $u$ vanishes outside $\overline{Q_R^+(X_1)}$, we have
\begin{equation}
(|\hat{f}|^p)_{Q_{\kappa r}^+(Y)}\le C(|f|^p)_{Q_{\kappa r}^+(Y)}+C_0I_z,\label{eq 4.106}
\end{equation}
where
$$I_z=\big(|(a_{ij}(t,z)-a_{ij}(t,x))\chi_{Q_R^+(X_1)}D_{ij}u|^p\big)_{Q_{\kappa r}^+(Y)},$$
and $\chi_{Q^+_R(X_1)}$ is the indicator function of $Q^+_R(X_1)$.
Denote $B^+$ to be $B_{\kappa r}^+(y)$ if $\kappa r<R$, or to be $B_R^+(\hat{x})$ otherwise. Define $Q^+$ in the same fashion.
We note that
\begin{equation}
							\label{eq0716_1}
\mu(Q^+) \le C(d,\theta) \mu\left( Q^+_{\kappa r}(Y) \right)
\end{equation}
provided that $Q_R^+(X_1) \cap Q_{\kappa r}^+(Y) \ne \emptyset$. It is obvious if $\kappa r < R$, i.e., $Q^+ = Q_{\kappa r}^+(Y)$. If $\kappa r \ge R$, the inequality is proved in Lemma \ref{lem0716}.

Combining \eqref{eq 4.104} and \eqref{eq 4.106} and taking the average of each term with respect to $z$ in $B^+$, we reach
\begin{align}\nonumber
&(|DD_1u- (DD_1u)_{Q_r^+(Y)}|^p)_{Q_r^+(Y)}\le C_0\kappa^{-p/2}(|DD_1u|^p)_{Q_{\kappa r}^+(Y)}\\
&\quad +C_0\kappa^{d+\theta+2}(|f|^p)_{Q_{\kappa r}^+(Y)}+C_0\kappa^{d+\theta+2}\dashint_{B^+}I_z\,\mu_d(dz).\label{eq 4.104b}
\end{align}
Since $u$ vanishes outside $\overline{Q_R^+(X_1)}$, by H\"older's inequality, we get
\begin{align}\nonumber
&\dashint_{B^+}I_z\,\mu_d(dz)\\\nonumber
&=\frac{1}{\mu(Q_{\kappa r}^+(Y))}\dashint_{B^+}\int_{Q_{\kappa r}^+(Y)\cap Q_R^+(X_1)}|(a_{ij}(t,z)-a_{ij}(t,x))D_{ij}u|^p\\
&\quad\cdot \mu(dx\,dt)\,\mu_d(dz)\nonumber\\
\nonumber
&\le \frac{1}{\mu(Q_{\kappa r}^+(Y))}\dashint_{B^+}\Big(\int_{Q^+}|a_{ij}(t,z)-a_{ij}(t,x)|^{\beta^\prime p}\,\mu(dx\,dt)\Big)^{1/\beta^\prime}\,\mu_d(dz)\\\label{eq 4.151}
&\quad \cdot \Big(\int_{Q_{\kappa r}^+(Y)\cap Q_R^+(X_1)}|D^2 u|^{\beta p}\,\mu(dx\,dt)\Big)^{1/\beta}.
\end{align}
By the boundedness of $a_{ij}$, H\"older's inequality, and the definition of osc$_x$, we have
\begin{align}\nonumber
&\dashint_{B^+}\Big(\int_{Q^+}|a_{ij}(t,z)-a_{ij}(t,x)|^{\beta^\prime p}\,\mu(dx\,dt)\Big)^{1/\beta^\prime}\,\mu_d(dz)\\\nonumber
&\le C_1\Big(\dashint_{B^+}\int_{Q^+}|a_{ij}(t,z)-a_{ij}(t,x)|
\,\mu(dx\,dt)\,\mu_d(dz)\Big)^{1/\beta^\prime}\\\label{eq 4.152}
&\le C_1 \left( \mu(Q^+)\,\text{osc}{_x}(a_{ij}, Q^+) \right)^{1/\beta^\prime},
\end{align}
where $C_1=C_1(d,\delta,p,\beta)$.  From \eqref{eq 4.151},  \eqref{eq 4.152}, and \eqref{eq0716_1}, we obtain
\begin{align}\label{eq 4.153}
\dashint_{B^+}I_z\,\mu_d(dz)\le C_1(A_R^\#)^{1/\beta^\prime}(|D^2 u|^{\beta p})_{Q_{\kappa r}^+(Y)}^{1/\beta}.
\end{align}
Combining \eqref{eq 4.104b} and \eqref{eq 4.153}, we get \eqref{eq 4.107}. The lemma is proved.
\end{proof}

Following exactly the proof of Lemma \ref{mean 1} with Lemma \ref{key lemma} in place of Lemma \ref{normal simple}, we obtain the lemma below.
\begin{lemma}\label{mean 2}
Let $R>0, \kappa \ge 32$, and $\beta, \beta^\prime\in (1, \infty)$ satisfying   $1/\beta+1/\beta^\prime=1$. Let $u\in C_{\text{loc}}^{\infty}(\overline{\bR^{d+1}_+})$ be compactly supported on $\overline{Q^+_R(X_1)}$, where $X_1\in \overline{\bR^{d+1}_+}$. Moreover, $D_1u=0$ on $\{x_1=0\}$ and $f:=-u_t+Lu$.
Then for any $r>0$ and $Y=(s,y)\in \overline{\bR^{d+1}_+}$, we have
\begin{align}\nonumber
&(|D_{x^\prime}^2u-(D_{x^\prime}^2u)_{Q_r^+(Y)}|^p)_{Q_r^+(Y)}\\\nonumber
&\le C_0\kappa^{-\alpha}(|D_{x^\prime}^2 u|^p)_{Q_{\kappa r}^+(Y)} + C_0\kappa^{d+\theta+2}(|f|^p+|DD_1u|^p)_{Q^+_{\kappa r}(Y)}\\
&\quad +C_1\kappa^{d+\theta+2}(A_R^\#)^{1/\beta^{\prime}}(|D^2 u|^{\beta p})_{Q_{\kappa r}^+(Y)}^{1/\beta},\nonumber
\end{align}
where $C_0=C_0(d,\delta, p,\theta)$, $C_1=C_1(d,\delta,p,\theta, \beta)$, and $\alpha=\alpha(d,\theta, p)>0$.
\end{lemma}

Next we recall the Hardy--Littlewood maximal function theorem and the Fefferman--Stein theorem on sharp functions. Let
$$\mathcal{Q}=\{ Q^+_r(z): z=(t,x)\in\overline{\bR^{d+1}_+}, r>0 \}.$$
For a function $g$ defined on $\bR^{d+1}_+$, the weighted (parabolic) maximal and sharp functions of $g$ are given by
\begin{align*}
\cM g(t,x)&=\sup_{Q\in \mathcal{Q}, (t,x)\in Q}\dashint_Q |g(s,y)|\,\mu(dy\,ds),\\
g^\#(t,x)&=\sup_{Q\in \mathcal{Q}, (t,x)\in Q}\dashint_Q |g(s,y)-(g)_Q|\,\mu(dy\,ds).
\end{align*}
For any $\theta>d-1$ and $g\in \bL_{p,\theta}(\bR^{d+1}_+)$,  we have
$$\|g\|_{\bL_{p,\theta}(\bR^{d+1}_+)}\le C\|g^\#\|_{\bL_{p,\theta}(\bR^{d+1}_+)},\quad \|\cM g\|_{\bL_{p,\theta}(\bR^{d+1}_+)}\le C\|g\|_{\bL_{p,\theta}(\bR^{d+1}_+)},$$
where $p\in (1,\infty)$ and $C=C(d,p,\theta)$. The first inequality above is known as the Fefferman--Stein theorem on sharp functions and the second one is the Hardy--Littlewood maximal function theorem, for instance see \cite[Chapter 3]{kry07}

Now we are ready to prove our main theorem.

\begin{proof}[Proof of Theorem \ref{main}]
By the method of continuity, it is enough to prove the a priori estimate \eqref{main estimate}.
Moreover, since the set of functions in $C^\infty(\overline{(-\infty,T) \times \bR^d_+})$ vanishing for large $(t,x)$ is dense in $W_{p,\theta}^{1,2}(-\infty,T)$, we only need to prove \eqref{main estimate} for infinitely differentiable functions with compact support.
In this case, the proof of \eqref{main estimate} can be divided into several steps.

Step 1: We consider $\lambda=0$, $b_i=c=0$, $T=\infty$, and $u\in C^\infty(\overline{\bR^{d+1}_+})$ vanishing outside $\overline{Q_{R_0}^+(X_1)}$ for some $X_1\in \overline{\bR^{d+1}}$, where $R_0$ is from Assumption \ref{assumption}. Let $\kappa\ge 32$ be a constant to be determined later. We fix $q\in (1,p)$ and $\beta\in(1,\infty)$, depending only on $p$ and $\theta$, such that $\beta q<p$ and $\theta<d-1+q$. Let $\beta^\prime$ be such that $1/\beta+1/\beta^\prime=1$. By applying Lemma \ref{mean 1} with $q$ in place of $p$ and using Assumption \ref{assumption}, we obtain
\begin{align*}
(DD_1 u)^\#(Y)\le& C_0 \kappa ^{-1/2}\cM^{1/q}(|DD_1u|^q)(Y)+C_0 \kappa^{(d+\theta+2)/q}\cM^{1/q}(|f|^q)(Y)\\
&+C_0\kappa^{(d+\theta+2)/q}\rho^{1/(\beta^\prime q)}\cM^{1/(\beta q)}(|D^2 u|^{\beta q})(Y)
\end{align*}
for any $Y\in \overline{\bR^{d+1}_+}$, where $C_0=C_0(d,\delta, p,\theta)$. This estimate, together with Fefferman--Stein theorem on sharp functions and the Hardy--Littlewood theorem on maximal functions, gives
\begin{align}\nonumber
&\|DD_1 u\|_{p, \theta}\le C\|(DD_1 u)^\#\|_{p, \theta}\\\nonumber
&\le C_0\kappa^{-1/2}\|\cM^{1/q}(|DD_1u|^q)\|_{p, \theta}+C_0\kappa^{(d+\theta+2)/q}\|\cM^{1/q}(|f|^q)\|_{p, \theta}\\\nonumber
&\quad +C_0\kappa^{(d+\theta+2)/q}\rho^{1/(\beta^\prime q)}\|\cM^{1/(\beta q)}(|D^2 u|^{\beta q})\|_{p,\theta}\\\label{eq 4.154}
&\le C_0\kappa^{-1/2}\|DD_1 u\|_{p,\theta}+C_0\kappa^{(d+\theta+2)/q}\|f\|_{p,\theta}+C_0\kappa^{(d+\theta+2)/q}\rho^{1/(\beta^\prime q)}\|D^2 u\|_{p,\theta}.
\end{align}
In the same way, we apply Lemma \ref{mean 2} instead of Lemma \ref{mean 1} to obtain the estimate of $D_{x^\prime}^2 u$:
\begin{align}\nonumber
\|D_{x^\prime}^2 u\|_{p,\theta}&\le C_0\kappa^{-\alpha/q}\|D_{x^\prime}^2 u\|_{p,\theta}+C_0\kappa^{(d+\theta+2)/q}(\|f\|_{p,\theta}+\|DD_1u\|_{p,\theta})\\\label{eq 4.155}
&\quad +C_0\kappa^{(d+\theta+2)/q}\rho^{1/(\beta^\prime q)}\|D^2 u\|_{p,\theta}.
\end{align}
By choosing $\kappa$ sufficiently large, from \eqref{eq 4.154} and \eqref{eq 4.155} we obtain
\begin{align}\label{eq 4.156}
&\|DD_1 u\|_{p, \theta}\le C_0\|f\|_{p,\theta}+C_0\rho^{1/(\beta^\prime q)}\|D^2 u\|_{p, \theta},\\\label{eq 4.157}
&\|D_{x^\prime}^2u\|_{p, \theta}\le C_0(\|f\|_{p,\theta}+\|DD_1 u\|_{p,\theta})+C_0\rho^{1/(\beta^\prime q)}\|D^2 u\|_{p, \theta}.
\end{align}
We combine \eqref{eq 4.156} and \eqref{eq 4.157} together to get
\begin{equation*}
\|D^2 u\|_{p,\theta}\le C_0\|f\|_{p,\theta}+C_0\rho^{1/(\beta^\prime q)}\|D^2 u\|_{p, \theta}.
\end{equation*}
Taking $\rho$ sufficiently small depending only on $d,\delta,p,\theta$, we arrive at
\begin{equation}
\|D^2u\|_{p,\theta}\le C_0\|f\|_{p,\theta}.\label{eq 4.161}
\end{equation}
On the other hand, since
$$u_t=a_{ij}D_{ij}u-f,$$
we have
\begin{equation}\label{eq 4.202}
\|u_t\|_{p,\theta}\le C_0\|f\|_{p,\theta}+C_0\|D^2 u\|_{p,\theta}.
\end{equation}
By Hardy's inequality (Lemma \ref{hardy}) with $r=\infty$,
\begin{equation}
\|M^{-1}D_1u\|_{p,\theta}\le C_0\|D_{11}u\|_{p,\theta} \label{eq 5.223}.
\end{equation}
Combining  \eqref{eq 4.161}, \eqref{eq 4.202}, and \eqref{eq 5.223}, we obtain \eqref{main estimate} for $\lambda = 0$.

Step 2: We remove the assumption that $u$ is compactly supported in $Q_{{R_0}}^+(X_1)$. By a standard partition of the unity argument with \eqref{eq 4.161}--\eqref{eq 5.223} (cf. \cite[Theorem 1.6.4]{kry07}) we see that
\begin{equation}
\|u_t\|_{p,\theta}+\|M^{-1}D_1u\|+\|D^2u\|_{p,\theta}\le C_0\|f\|_{p,\theta}+{C_1}\|u\|_{p,\theta},\label{eq 4.162}
\end{equation}
where $C_0 = C_0(d,\delta,p,\theta)$ and $C_1 = C_1(d,\delta,p,\theta,R_0)$.

Step 3: We still assume that $b_i = c = 0$ and $T=\infty$, but $\lambda$ is not necessarily zero. In this case, we follow an idea of S. Agmon.
Since $$-u_t+a_{ij}D_{ij}u=f-\lambda u,$$
by \eqref{eq 4.162} we have
\begin{align*}
\|u_t\|_{p,\theta}+\|M^{-1}D_1u\|_{p,\theta}+\|D^2 u\|_{p,\theta}&\le C_0\|f-\lambda u\|_{p,\theta}+C_1\|u\|_{p,\theta}\\
&\le C_0\|f\|_{p,\theta}+(C_0\lambda+C_1)\|u\|_{p,\theta}.
\end{align*}
Hence it is sufficient to show that  for large $\lambda$
$$\lambda\|u\|_{p,\theta}\le C_0\|f\|_{p,\theta}.$$

We pick a function $\zeta(y)\in C_0^\infty(\bR), \zeta\not\equiv 0$ and introduce the following notation
$$z=(x,y), \quad \hat{u}(t,z)=u(t,x)\zeta(y)\cos(\sqrt{\lambda}y),\quad \hat{L}{u}=L(t,x)u(t,z)+u_{yy}(t,z).$$
Finally, set $\bR^{d+2}_+=\{(t,z): x_1> 0\}$ and
$$\mathbb{B}_r(z)=(x_1-r,x_1+r)\times B_r(x^\prime,y), \quad \mathbb{Q}_r(t, z)=(t-r^2, t)\times \mathbb{B}_r(z),$$
$$\mathbb{B}^+_r(z)=\mathbb{B}_r(z)\cap \{x_1> 0\},\quad  \mathbb{Q}^+_r(t, z)=(t-r^2, t)\times \mathbb{B}^+_r(z).$$
For any $r>0$, $Z=(t,z)\in \overline{\bR^{d+2}_+}$, and $\hat{a}(t)$, we have
\begin{equation*}
\dashint_{\mathbb{B}^+_r(z)}|a_{ij}(t,x)-\hat{a}(t)|\,\mu_d(dx)\,dy\le C\dashint_{B_r^+(x)}|a_{ij}(t,x)-\hat{a}(t)|\,\mu_d(dx),
\end{equation*}
where $C=C(d)$.
In particular, by the definition of the osc$_x$ and setting
$$\hat{a}(t)=\dashint_{B^+_r(x)}a_{ij}(t,x)\,\mu_d(dx),$$
we have $\text{osc}_z(a_{ij}, \mathbb{Q}^+_r(t,z))\le C\,\text{osc}_x(a_{ij}, Q^+_r(t,x))$.
By a simple calculation,
\begin{equation*}
\hat{L}\hat{u}=f\cos(\sqrt{\lambda}y)\zeta(y)+u\zeta'' \cos(\sqrt{\lambda}y)-2\sqrt{\lambda}u\zeta' \sin(\sqrt{\lambda}y):=\hat{f}.
\end{equation*}
By \eqref{eq 4.162} with $\hat{u}$ instead of $u$ in dimension $d+2$,
\begin{equation}
\vertiii{D^2\hat{u}}_{p,\theta}\le C_0\vertiii{\hat{f}}_{p,\theta} + C_1\vertiii{\hat{u}}_{p,\theta},\label{eq 4.163}
\end{equation}
where $\vertiii{\cdot}_{p,\theta}$ is the weighted $L_p$ norm in $\bR^{d+2}_+$ with respect to $\mu(dx\,dt)\,dy$.  Note that for $\lambda>1$,
\begin{equation}
c_1\le \int_{\bR}|\zeta(y) \cos(\sqrt{\lambda}y)|^p\,dy:=A\le c_2,\label{eq 4.224}
\end{equation}
where $c_1$ and $c_2$ are positive constants independent of $\lambda$. Moreover, $A$ is bounded from above with $\zeta$ replaced by any derivatives of $\zeta$ and $\cos(\sqrt{\lambda}y)$ replaced by $\sin(\sqrt{\lambda}y)$.  Then from \eqref{eq 4.163}, we have
\begin{equation}
\vertiii{D_{yy}\hat{u}}_{p,\theta}\le \vertiii{D^2\hat{u}}_{p,\theta}\le C_0\|f\|_{p,\theta}+C_1(1+\sqrt\lambda)\|u\|_{p,\theta}.
                    \label{eq 4.164}
\end{equation}
Since
\begin{align}\nonumber
D_{yy}\hat{u}&=u(t,x)\zeta''(y)\cos(\sqrt{\lambda}y)-2\sqrt{\lambda}u(t,x)\zeta'(y)
\sin(\sqrt{\lambda}y)\\
&\quad -\lambda u(t,x)\zeta(y)\cos(\sqrt{\lambda} y),\label{eq 5.225}
\end{align}
Combining \eqref{eq 4.224}, \eqref{eq 4.164}, and \eqref{eq 5.225}, we have \begin{equation}
                                    \label{eq1.42}
\lambda\|u\|_{p,\theta}\le C_0\|f\|_{p,\theta}+C_1(1+\sqrt{\lambda})\|u\|_{p,\theta}.
\end{equation}
After choosing $\lambda$ sufficiently large depending on $d,\delta, p, \theta, R_0$ to absorb the term of $u$ on the right-hand side of \eqref{eq1.42} to the left-hand side, we have
 \begin{equation*}
\lambda\|u\|_{p,\theta}\le C_0\|f\|_{p,\theta},
\end{equation*}
where $C_0=C_0(d,\delta,p,\theta)$.
By the interpolation inequality
$$\sqrt{\lambda}\|D u\|_{p,\theta}\le C_0\|D^2 u\|_{p,\theta}+C_0\lambda\|u\|_{p,\theta},$$
we finish the proof of Step 3.

Step 4: We remove the assumption that $b_i=c=0$ by moving the terms of $b_i$ and $c$ to the right-hand side
\begin{equation*}
-u_t+a_{ij}D_{ij}u-\lambda u=f-b_iD_i u-cu.
\end{equation*}
By the conclusion in Step 3 with $b_i=c=0$, there exists $\lambda_0=\lambda_0(d,\delta,p,\theta,R_0)$ such that for any $\lambda\ge \lambda_0$,
\begin{align*}
&\|u_t\|_{p,\theta}+\lambda\|u\|_{p,\theta}+\sqrt{\lambda}\|Du\|_{p,\theta}+\|M^{-1}D_1u\|_{p,\theta}+\|D^2 u\|_{p,\theta}\\
&\le C_0\|f-b_iD_i u-cu\|_{p,\theta}\\
&\le C_0\|f\|_{p,\theta}+C_0K\|D u\|_{p,\theta}+C_0K\|u\|_{p,\theta}.
\end{align*}
By taking $\lambda$ sufficiently large depending on $d$, $\delta$, $p$, $\theta$,  $R_0$, and $K$, we get
\begin{equation*}
\|u_t\|_{p,\theta}+\lambda\|u\|_{p,\theta}+\sqrt{\lambda}\|D u\|_{p,\theta}+\|M^{-1}u\|_{p,\theta}+\|D^2 u\|_{p,\theta}\le C_0\|f\|_{p,\theta}.
\end{equation*}

Step 5: To remove the assumption that $T=\infty$,  we simply follow the standard step in \cite[Theorem 2.1]{DK14} or \cite[Theorem 6.4.1]{kry07}. Therefore, the estimate \eqref{main estimate} is proved.

Finally, in the case when $b_i=c=0$ and $a_{ij}=a_{ij}(t)$, by using a scaling argument we can take $R_0=0$. The theorem is proved.
\end{proof}

\bibliographystyle{plain}
\bibliography{elliptic_pde}

\def\cprime{$'$} \def\cprime{$'$}
\begin{thebibliography}{10}

\bibitem{CFL91}
Filippo Chiarenza, Michele Frasca, and Placido Longo.
\newblock Interior {$W^{2,p}$} estimates for nondivergence elliptic equations
  with discontinuous coefficients.
\newblock {\em Ricerche Mat.}, 40(1):149--168, 1991.

\bibitem{CFL93}
Filippo Chiarenza, Michele Frasca, and Placido Longo.
\newblock {$W^{2,p}$}-solvability of the {D}irichlet problem for nondivergence
  elliptic equations with {VMO} coefficients.
\newblock {\em Trans. Amer. Math. Soc.}, 336(2):841--853, 1993.

\bibitem{Dong11}
H.~Dong.
\newblock Parabolic equations with variably partially {VMO} coefficients.
\newblock {\em Algebra i Analiz}, 23(3):150--174, 2011.

\bibitem{DK14}
H.~{Dong} and D.~{Kim}.
\newblock {Elliptic and parabolic equations with measurable coefficients in
  weighted Sobolev spaces}.
\newblock {\em ArXiv e-prints}, March 2014.

\bibitem{HZ14}
H.~{Dong} and H.~{Zhang}.
\newblock {Conormal problem of higher-order parabolic systems}.
\newblock {\em ArXiv e-prints}, January 2014.

\bibitem{Dong10}
Hongjie Dong.
\newblock Solvability of parabolic equations in divergence form with partially
  {BMO} coefficients.
\newblock {\em J. Funct. Anal.}, 258(7):2145--2172, 2010.

\bibitem{DK11}
Hongjie Dong and Doyoon Kim.
\newblock {$L_p$} solvability of divergence type parabolic and elliptic systems
  with partially {BMO} coefficients.
\newblock {\em Calc. Var. Partial Differential Equations}, 40(3-4):357--389,
  2011.

\bibitem{DK11ARMA}
Hongjie Dong and Doyoon Kim.
\newblock On the {$L_p$}-solvability of higher order parabolic and elliptic
  systems with {BMO} coefficients.
\newblock {\em Arch. Ration. Mech. Anal.}, 199(3):889--941, 2011.

\bibitem{Kimdo07}
Doyoon Kim.
\newblock Parabolic equations with measurable coefficients. {II}.
\newblock {\em J. Math. Anal. Appl.}, 334(1):534--548, 2007.

\bibitem{KimKimLee14}
Ildoo Kim, Kyeong-Hun Kim, and Kijung Lee.
\newblock A weighted {$L_p$}-theory for divergence type parabolic {PDE}s with
  {BMO} coefficients on {$C^1$}-domains.
\newblock {\em J. Math. Anal. Appl.}, 412(2):589--612, 2014.

\bibitem{MR3174219}
Kyeong-Hun Kim.
\newblock A weighted {S}obolev space theory of parabolic stochastic {PDE}s on
  non-smooth domains.
\newblock {\em J. Theoret. Probab.}, 27(1):107--136, 2014.

\bibitem{KimKry04}
Kyeong-Hun Kim and N.~V. Krylov.
\newblock On the {S}obolev space theory of parabolic and elliptic equations in
  {$C^1$} domains.
\newblock {\em SIAM J. Math. Anal.}, 36(2):618--642, 2004.

\bibitem{KimLee13}
Kyeong-Hun Kim and Kijung Lee.
\newblock A weighted {$L_p$}-theory for parabolic {PDE}s with {BMO}
  coefficients on {$C^1$}-domains.
\newblock {\em J. Differential Equations}, 254(2):368--407, 2013.

\bibitem{KozNaz14}
V.~{Kozlov} and A.~I. {Nazarov}.
\newblock {Oblique derivative problem for non-divergence parabolic equations
  with discontinuous in time coefficients}.
\newblock {\em ArXiv e-prints}, January 2013.

\bibitem{KozNaz09}
Vladimir Kozlov and Alexander Nazarov.
\newblock The {D}irichlet problem for non-divergence parabolic equations with
  discontinuous in time coefficients.
\newblock {\em Math. Nachr.}, 282(9):1220--1241, 2009.

\bibitem{Kry94}
N.~V. Krylov.
\newblock A {$W^n_2$}-theory of the {D}irichlet problem for {SPDE}s in general
  smooth domains.
\newblock {\em Probab. Theory Related Fields}, 98(3):389--421, 1994.

\bibitem{Kry99}
N.~V. Krylov.
\newblock Weighted {S}obolev spaces and {L}aplace's equation and the heat
  equations in a half space.
\newblock {\em Comm. Partial Differential Equations}, 24(9-10):1611--1653,
  1999.

\bibitem{Kry97}
N.~V. Krylov.
\newblock Parabolic and elliptic equations with {VMO} coefficients.
\newblock {\em Comm. Partial Differential Equations}, 32(1-3):453--475, 2007.

\bibitem{kry07}
N.~V. Krylov.
\newblock {\em Lectures on elliptic and parabolic equations in {S}obolev
  spaces}, volume~96 of {\em Graduate Studies in Mathematics}.
\newblock American Mathematical Society, Providence, RI, 2008.

\bibitem{Kry09}
N.~V. Krylov.
\newblock On divergence form {SPDE}s with {VMO} coefficients in a half space.
\newblock {\em Stochastic Process. Appl.}, 119(6):2095--2117, 2009.

\bibitem{KrLo99b}
N.~V. Krylov and S.~V. Lototsky.
\newblock A {S}obolev space theory of {SPDE}s with constant coefficients in a
  half space.
\newblock {\em SIAM J. Math. Anal.}, 31(1):19--33, 1999.

\bibitem{KrLo99a}
N.~V. Krylov and S.~V. Lototsky.
\newblock A {S}obolev space theory of {SPDE}s with constant coefficients on a
  half line.
\newblock {\em SIAM J. Math. Anal.}, 30(2):298--325, 1999.

\bibitem{Kry07JFA}
N.V. Krylov.
\newblock Parabolic equations with \text{VMO} coefficients in sobolev spaces
  with mixed norms.
\newblock {\em Journal of Functional Analysis}, 250(2):521 -- 558, 2007.

\bibitem{Kuf85}
Alois Kufner.
\newblock {\em Weighted {S}obolev spaces}.
\newblock A Wiley-Interscience Publication. John Wiley \& Sons, Inc., New York,
  1985.
\newblock Translated from the Czech.

\bibitem{Nad97}
Nikolai Nadirashvili.
\newblock Nonuniqueness in the martingale problem and the {D}irichlet problem
  for uniformly elliptic operators.
\newblock {\em Ann. Scuola Norm. Sup. Pisa Cl. Sci. (4)}, 24(3):537--549, 1997.

\bibitem{Ura67}
N.~N. Ural{\cprime}ceva.
\newblock The impossibility of {$W_{q}{}^{2}$} estimates for multidimensional
  elliptic equations with discontinuous coefficients.
\newblock {\em Zap. Nau\v cn. Sem. Leningrad. Otdel. Mat. Inst. Steklov.
  (LOMI)}, 5:250--254, 1967.

\end{thebibliography}

\def\cprime{$'$}\def\cprime{$'$} \def\cprime{$'$} \def\cprime{$'$}
  \def\cprime{$'$} \def\cprime{$'$}

\end{document}